\documentclass[11pt]{article}
\usepackage{amsmath}
\usepackage{dsfont}
\usepackage{mathrsfs}
\usepackage{amsmath,amssymb}
\usepackage{amsfonts}
\usepackage{hyperref}
\usepackage{amsthm}
\usepackage{graphicx}
\usepackage{subfigure}
\usepackage{xcolor}
\renewcommand{\qed}{\hfill\small{$\square$}\normalsize}

\hfuzz=\maxdimen
\theoremstyle{definition}
\newtheorem{lemma}{Lemma}[section]
\newtheorem{definition}[lemma]{Definition}
\newtheorem{proposition}[lemma]{Proposition}
\newtheorem{theorem}[lemma]{Theorem}
\newtheorem{corollary}[lemma]{Corollary}

\newtheorem{remark}{Remark}

\newtheorem{claim}{Claim}

\numberwithin{equation}{section}

\renewcommand{\qed}{\hfill\small{$\square$}\normalsize}

\DeclareFixedFont{\Acknowledgment}{OT1}{cmr}{bx}{n}{14pt}
\begin{document}

\title{\bf On the deformation of inversive distance circle packings}
\author{Huabin Ge, Wenshuai Jiang}
\maketitle

\begin{abstract}
In this paper, we generalize Chow-Luo's combinatorial Ricci flow to inversive distance circle packing setting. Although the solution to the generalized flow may develop singularities in finite time, we can always extend the solution so as it exists for all time and converges exponentially fast. Thus the generalized flow can be used to deform any inversive distance circle packing to a unique packing with prescribed cone angle. We also give partial results on the range of all admissible cone angles, which generalize the classical Andreev-Thurston's theorem.
\end{abstract}


\section{Introduction}\label{Introduction}
In his work on constructing hyperbolic metrics on three dimensional manifolds, Thurston \cite{T1} introduced patterns of circles on triangulated surfaces. In the pioneered work of Chow and Luo \cite{CL1}, they first established an intrinsic connection between Thurston's circle packing and surface Ricci flow. In fact, Chow and Luo introduced an analog of Hamilton's Ricci flow in the combinatorial setting, which converges exponentially fast to Thurston's circle packing on surfaces. As
a consequence, they obtained a new proof of Thurston's existence of circle packing theorem and a new algorithm to find circle packings. However, Thurston's circle packing requires adjacent circles intersect with each other, which is too restrictive. Hence Bowers and Stephenson \cite{Bowers-Stephenson} introduced inversive distance circle packing, which allow adjacent circles to be disjoint and measure their relative positions by the inversive distance. Bowers-Stephenson's relaxation of intersection condition is very useful for practical applications, especially in medical imaging and computer graphics fields, see Hurdal and Stephenson \cite{Hurdal-Stephenson} for example. In this paper, we will generalize Chow-Luo's combinatorial Ricci flow to Bowers-Stephenson's inversive distance circle packing setting. The main idea in this paper comes from readings of Bobenko, Pincall and Springborn \cite{Bobenko}, Guo \cite{Guoren}, Luo \cite{Luo1}, Glickenstein \cite{Glickenstein1}, Colin de Verdi$\grave{e}$re \cite{Colindev}, Rivin \cite{Ri} and Marden and Rodin \cite{Marden-Rodin}. We follow the approach pioneered by Luo \cite{Luo1}.

Suppose $M$ is a closed surface with a triangulation $\mathcal{T}=\{V,E,F\}$,
where $V,E,F$ represent the sets of vertices, edges and faces respectively.
Let $\Phi: E\rightarrow [0,\frac{\pi}{2}]$ be a function assigning each edge $\{ij\}$ a weight $\Phi_{ij}\in [0,\frac{\pi}{2}]$.
The triple $(M, \mathcal{T}, \Phi)$ will be referred to as a weighted triangulation of $M$ in the following.
All the vertices are ordered one by one, marked by $1, \cdots, N$, where $N=V^\sharp$
is the number of vertices. Throughout this paper, all functions $f: V\rightarrow \mathds{R}$ will be regarded as column
vectors in $\mathds{R}^N$ and $f_i$ is the value of $f$ at $i$. Each map $r:V\rightarrow (0,+\infty)$ is called a circle packing metric.
Given $(M, \mathcal{T}, \Phi)$, we equip each edge $\{ij\}$ with a length
\begin{equation}\label{definition of length of edge}
l_{ij}^\Phi=\sqrt{r_i^2+r_j^2+2r_ir_j\cos \Phi_{ij}}.
\end{equation}
Thurston claimed \cite{T1} that, for each face $\{ijk\}\in F$, the three lengths $\{l_{ij}^\Phi, l_{jk}^\Phi, l_{ik}^\Phi\}$ satisfy the triangle inequalities.
Thus the triangulated surface $(M, \mathcal{T}, \Phi)$ could be taken as gluing many Euclidean triangles coherently. Suppose $\theta_i^{jk}$ is the inner angle of the triangle $\{ijk\}$ at the vertex $i$, the classical well-known discrete Gaussian curvature at each vertex $i$ is
\begin{equation}\label{classical Gauss curv}
K_i=2\pi-\sum_{\{ijk\} \in F}\theta_i^{jk},
\end{equation}
and the discrete Gaussian curvature $K_i$ satisfies the following discrete version of Gauss-Bonnet formula \cite{CL1}:
\begin{equation}\label{Gauss-Bonnet}
\sum_{i=1}^NK_i=2\pi \chi(M).
\end{equation}

By the discrete Gauss-Bonnet formula (\ref{Gauss-Bonnet}), the average of total discrete Gaussian curvature $K_{av}=\sum_{i=1}^NK_i/N$ is determined only by the topological and combinatorial information of $M$; that is
\begin{equation}
K_{av}=\frac{2\pi \chi(M)}{N}.
\end{equation}
Different from the smooth surface case, the constant curvature metric, i.e., a metric $r$ with $K_i=K_{av}$ for all $i\in V$, does not always exist. Besides the topological structure, the combinatorial structure plays an essential role for the existence of constant curvature metric (see Theorem \ref{Thm-thurston} in Section \ref{section-andreev-thurston}). Moreover, if the constant curvature metric exists, it is unique up to a scalar multiplication. This is a deep result implied in Andreev-Thurston theorem which is stated in detail in Section \ref{section-andreev-thurston} (see Theorem \ref{Thm-Andreev-thurston}). For a fixed circle packing, it is obviously that the curvatures are completely determined by these circle radii. On the contrary, the Andreev-Thurston theorem answers the question about the existence and uniqueness of radii of circle patterns when curvatures are given. The uniqueness, or say ``rigidity" result, says that the circle packing is uniquely determined by curvatures. While the existence result says that all possible curvatures form an open convex polytope.

In order to deform Thurston's circle packing metrics, Chow and Luo \cite{CL1} introduced a combinatorial Ricci flow
\begin{equation}\label{Def-ChowLuo's flow}
\begin{cases}
{r_i}'(t)=-K_ir_i \\
\;\,r(0)\in \mathds{R}^N_{>0}.
\end{cases}
\end{equation}
and its normalization
\begin{equation}\label{Def-ChowLuo's normalized flow}
\begin{cases}
{r_i}'(t)=(K_{av}-K_i)r_i\\
\;\,r(0)\in \mathds{R}^N_{>0}.
\end{cases}
\end{equation}
They proved
\begin{theorem}\label{Thm-Chow-Luo} \textbf{(Chow-Luo)} \;For any initial metric $r(0)$, the solution to flow (\ref{Def-ChowLuo's normalized flow}) exists for all time.
Additionally, flow (\ref{Def-ChowLuo's normalized flow}) converges if and only if there exists a metric of constant curvature. Furthermore, if the solution converges, then it converges exponentially fast to the metric of constant curvature.
\end{theorem}

The following of this paper contains two parts. In the first part, we shall extend Chow-Luo's Theorem \ref{Thm-Chow-Luo} for circle packing metric to Bowers-Stephenson＊s inversive distance circle packing setting. This part occupies Section \ref{ICP-metric} and Section \ref{section-prove}.
In the second part, we shall state partial results regarding to the extension of Andreev-Thurston theorem to inversive distance setting in Section \ref{section-andreev-thurston}.

\section{Inversive distance circle packing metric}\label{ICP-metric}
We begin this section by briefly recalling the inversive distance in Euclidean geometry. See Bowers-Hurdal \cite{Bowers-Hurdal} and Bowers-Stephenson \cite{Bowers-Stephenson} for more detailed discussions. In the Euclidean plane, consider two circles $c_1$, $c_2$ with radii $r_1$, $r_2$ respectively, and assume that $c_1$ does not contain $c_2$ and vice versa. If the distance between their center is $l_{12}$, then the inversive distance between $c_1$, $c_2$ is given by the formula
\begin{equation*}
I(c_1, c_2)=\cfrac{l_{12}^2-r_1^2-r_2^2}{2r_1r_2}.
\end{equation*}
If one considers the Euclidean plane as appeared in the infinity of the hyperbolic 3-space $\mathds{H}^3$, then $c_1$ and $c_2$ are the boundaries of two totally geodesic hyperplanes $D_1$ and $D_2$. The inversive distance $I(c_1, c_2)$ is essentially the hyperbolic distance (or the intersection angle) between $D_1$ and $D_2$, which is invariant under the inversion and hence the name.

Note that $l_{12}> |r_1-r_2|$, we have $-1<I(c_1, c_2)< +\infty$. The inverse distance $I(c_1, c_2)$ describes the relative positions of $c_1$ and $c_2$. When $I(c_1, c_2)\in (-1, 0)$, the circles $c_1$, $c_2$ intersect with an intersection angle $\arccos I(c_1, c_2)\in (\frac{\pi}{2}, \pi)$. When $I(c_1, c_2)\in [0, 1]$, the circles $c_1$, $c_2$ intersect with an intersection angle $\arccos I(c_1, c_2)\in [0, \frac{\pi}{2}]$. When $I(c_1, c_2)\in (1, +\infty)$, the circles $c_1$, $c_2$ are separated.

We then reformulate Bowers and Stephenson's construction of an inversive distance circle packing with prescribed inversive distance $I$ on triangulated surface $(M, \mathcal{T})$. We consider $I$ as a function defined on all edges; that is $I: E\rightarrow (-1, +\infty)$, and we call $I$ the inversive distance. Let $I>c$ means $I_{ij}>c$ for each edge $\{ij\}\in E$, and $I\geq c$ means $I_{ij}\geq c$ for each edge $\{ij\}\in E$. Now fix $(M, \mathcal{T}, I)$ with $I>-1$. For every given radius vector $r\in \mathds{R}^N_{>0}$, we equip each edge $\{ij\}\in E$ with a length
\begin{equation}\label{Def-edge-length}
l_{ij}=\sqrt{r_i^2+r_j^2+2r_ir_jI_{ij}}.
\end{equation}
However, different with Thurston's construction of $l_{ij}^\Phi$ in formula (\ref{definition of length of edge}), for a face $\{ijk\}\in F$, the three lengths $\{l_{ij}, l_{jk}, l_{ik}\}$ may not satisfy the triangle inequalities any more. If for each $\{ijk\}\in F$, the three lengths $\{l_{ij}, l_{jk}, l_{ik}\}$ all satisfy the triangle inequalities, then the triangulated surface $(M, \mathcal{T}, I)$ with edge lengthes $l_{ij}$ forms an Euclidean polyhedral surface. In this case, the corresponding radius vector $r\in \mathds{R}^N_{>0}$ is called a inversive distance circle packing metric. In the following, we use $\Omega$ to represent the space of all possible inversive distance circle packing metrics, i.e.,
\begin{equation}
\Omega=\Big\{r\in \mathds{R}^N_{>0}\;\big|\;l_{ij}+l_{jk}>l_{ik},\;l_{ij}+l_{ik}>l_{jk},\;l_{ik}+l_{jk}>l_{ij}, \;\forall \;\{ijk\}\in F\Big\}.
\end{equation}
\begin{remark}
$\Omega$ is a simply connected open cone of $\mathds{R}^N_{>0}$. However, $\Omega$ is generally not convex and may be empty.
\end{remark}

We want to deform a inversive distance circle packing. It's enough to deform all the radii of circles in the packing, i.e., the metric $r$ in $\Omega$. According to Chow-Luo's idea, we deform the metric $r$ along the following inversive distance combinatorial Ricci flow
\begin{equation}\label{Def-flow}
\begin{cases}
{r_i}'(t)=-K_ir_i \\
\;\,r(0)\in \Omega
\end{cases}
\end{equation}
and its normalization
\begin{equation}\label{Def-normalized-flow}
\begin{cases}
{r_i}'(t)=(K_{av}-K_i)r_i \\
\;\,r(0)\in \Omega
\end{cases}.
\end{equation}

Chow and Luo's combinatorial Ricci flow in the inversive distance circle packing setting first appeared in \cite{Zhang} in a unified form with other types of discrete curvature flows. However, they didn't give further convergence properties for this flow. We shall study this flow carefully in this paper. Flow (\ref{Def-flow}) and the normalized flow (\ref{Def-normalized-flow}) differ only by a change of scale in space. Namely $r_i(t)$ is a solution to (\ref{Def-flow}) if and only if $e^{\frac{2\pi\chi(M)t}{N}}r_i(t)$ is a solution of (\ref{Def-normalized-flow}). Furthermore, if $r_i(t)$ is a solution to the normalized equation (\ref{Def-normalized-flow}), the product $\Pi_{i=1}^N r_i(t)$ is a constant. In the following, we mainly study the behavior of the normalized flow (\ref{Def-normalized-flow}). Set $\ln\Omega=\{x\in\mathds{R}^N|(e^{x_1},\cdots,e^{x_N})\in \Omega\}$. Using a coordinate change $u_i=\ln r_i$, we change flow (\ref{Def-normalized-flow}) to an autonomous ODE system
\begin{equation}\label{Def-normalized-flow-u}
\begin{cases}
{u_i}'(t)=K_{av}-K_i\\
\;\,u(0)\in \ln\Omega
\end{cases}.
\end{equation}

Note that, $K_i$ as a function of $u=(u_1,\cdots,u_N)^T$ is smooth and hence locally Lipschitz continuous. By Picard theorem in classical ODE theory, flow (\ref{Def-normalized-flow-u}) has a unique solution $r(t)$, $t\in[0, \epsilon)$ for some $\epsilon>0$.
\begin{proposition}\label{Prop-nonsingularity}
Given a triangulated surface $(M, \mathcal{T}, I)$ with inversive distance $I\geq0$. Let $r(t)$ be the solution to flow (\ref{Def-normalized-flow}), then for each $i\in V$, $r_i(t)$ can not go to zero or infinity in any finite time interval $[0, a)$ with $a<+\infty$.
\end{proposition}
\begin{proof}
Note that, $|K_{av}-K_i|$ are uniformly bounded by a constant $c>0$, which depends only on the triangulation. Hence
$$r_i(0)e^{-ct}\leq r_i(t)\leq r_i(0)e^{ct},$$
which implies that $r_i(t)$ can not go to zero or infinity in finite time.\qed
\end{proof}

Suppose $\{r(t)|0\leq t<T\}$ is the unique solution to flow (\ref{Def-normalized-flow}) on a right maximal time interval $[0, T)$ with $0<T\leq +\infty$.
If $T<+\infty$, then $r(t)$ touches the boundary of $\Omega$ as $t\uparrow T$. By Proposition \ref{Prop-nonsingularity}, $c\leq r_i(t)\leq C$ for all $i\in V$ and $t\in[0,T)$, where $c$ and $C$ are positive constants. Hence all edges $l_{ij}(t)$ remain positive for $t\in[0,T)$, and there exists a sequence of time $t_n\uparrow T$ and a triangle $\{ijk\}\in F$, such that the triangle inequality in triangle $\{ijk\}\in F$ do not hold any more as $n\rightarrow +\infty$. However, we can always extend the solution $r(t)$ so that it exists for all time $t\in[0,+\infty)$. The basic idea is to extend the definition of curvature $K$ continuously to a generalized curvature $\widetilde{K}$, which is defined on for all $r\in \mathds{R}^N_{>0}$ (for details, see formula (\ref{def-K-tuta}) in section \ref{section-prove}). Even the triangle inequalities are not satisfied, $\widetilde{K}$ is still well defined and uniformly bounded by topological and combinatorial data. As to flow (\ref{Def-normalized-flow}), even if the triangular inequalities may not valid in some finite time, we can still deform the inversive distance metric along an extended flow until time tends to $+\infty$. We shall prove the following extension theorem in section \ref{section-prove}:

\begin{theorem}\label{Thm-main-1}
Given a triangulated surface $(M, \mathcal{T}, I)$ with inversive distance $I\geq 0$. Suppose $\{r(t)|t\in[0,T)\}$ is the unique maximal solution to flow (\ref{Def-normalized-flow}) with $0<T\leq +\infty$. Then we can always extend it to a solution $\{r(t)|t\in[0,+\infty)\}$ when $T<+\infty$. This is equivalent to say, the solution to the extended flow
\begin{equation}\label{Def-extended-flow-u}
\begin{cases}
{u_i}'(t)=K_{av}-\widetilde{K}_i\\
\;\,u(0)\in \ln\Omega
\end{cases}
\end{equation}
exists for all time $t\in[0,+\infty)$.
\end{theorem}

Next we see the convergence behavior. If the solution $u(t)$ to flow (\ref{Def-normalized-flow-u}) exists for all time $t\in[0, +\infty)$ and converges to a metric $u^*\in \ln\Omega$ (or more generally, the extended solution $u(t)$ with $t\in[0, +\infty)$ converges to $u^*$), then $u^*$ must be a critical point of the ODE system (\ref{Def-normalized-flow-u}). Thus $u^*$ is a zero point of equation $K_{av}-K=0$, and $u^*$ (or say the corresponding $r^*$) is a metric of constant curvature. As a consequence, the metric of constant curvature exists in $\ln\Omega$. Furthermore, in this case, flow (\ref{Def-normalized-flow}) can be used to deform the metric $r$ to a metric of constant curvature. This shows that the existence of constant curvature metric is a necessary condition for the convergence of solution $r(t)$. However, this condition is essentially enough. We will prove

\begin{theorem}\label{Thm-main-2}
Given a triangulated surface $(M, \mathcal{T}, I)$ with inversive distance $I\geq 0$. Assuming there exists a metric of constant curvature $r^*\in\Omega$. Then $r(t)$ can always be extended to a solution that converges exponentially fast to a metric of constant curvature as $t\rightarrow+\infty$. That is, the solution to the extended flow
(\ref{Def-extended-flow-u}) converges exponentially fast to a metric of constant curvature as $t\rightarrow+\infty$.
\end{theorem}

It is remarkable that the extension phenomenon is still true for Luo's combinatorial Yamabe flow \cite{Luo0}, see \cite{Ge-Jiang1} for details. We shall prove Theorem \ref{Thm-main-1} and Theorem \ref{Thm-main-2} in the following section.

\section{Deform the metric to constant curvature}\label{section-prove}
In the following, if we write $(M, \mathcal{T}, \Phi)$, we mean a triangulated surface with weight $\Phi\in[0,\frac{\pi}{2}]$ setting, and we consider circle packing metric $r\in \mathds{R}^N_{>0}$. If we write $(M, \mathcal{T}, I)$, we mean a triangulated surface with inversive distance $I\geq 0$ setting, and we consider inversive distance circle packing metric $r\in \Omega$. We first see what happens in a single triangle settings. Assuming a triangle $\triangle123$ is configured by three circles with three fixed non-negative numbers $I_{12}$, $I_{23}$ and $I_{13}$ as inversive distances. For any $(r_1, r_2, r_3)^T\in \mathds{R}^3_{>0}$, three edge lengths $l_{12}$, $l_{23}$ and $l_{13}$ are defined by formula (\ref{Def-edge-length}). Set
\begin{equation}\label{Def-delta}
\Delta=\Big\{(r_1, r_2, r_3)^T\in \mathds{R}^3_{>0}\;\big|\;l_{12}+l_{23}>l_{13},\;l_{12}+l_{13}>l_{23},\;l_{13}+l_{23}>l_{12}\Big\}.
\end{equation}
It was shown that $\Delta$ is a simply connected open cone-like subset of $\mathds{R}^3_{>0}$, and for more analysis of the shape of $\Delta$, see
\cite{Guoren,Luo1}. $\Delta$ is never empty. In fact, assuming $0\leq I_{12}\leq I_{13}\leq I_{23}$, and let $r_1=r_2=1$, $r_3=r$ with $r>0$ satisfying equation $r^2+2rI_{13}=1+2I_{12}$, then $(1,1,r)^T\in\Delta$. For each $(r_1, r_2, r_3)^T\in\Delta$, there corresponds an Euclidean triangle $\triangle 123$, with three edge lengths $l_{12}$, $l_{23}$ and $l_{13}$. Denote $\{ij\}$ as the edge that has length $l_{ij}$, for each $i, j\in \{1, 2, 3\}$ and $i\neq j$. Denote
$1$, $2$ and $3$ as three vertices that face the edge $\{23\}$, $\{13\}$ and $\{12\}$ respectively. Moreover, let $\theta_i$ be the inner angle at vertex $i$ for each $i\in\{1, 2, 3\}$. Thus we get the angle map $\theta=(\theta_1,\theta_2,\theta_3)^T: \Delta\rightarrow \mathds{R}^3_{>0}$.

\begin{lemma}\label{Lemma-Guo}(Guo \cite{Guoren})
For each $(r_1, r_2, r_3)^T\in\Delta$, let $u_i=\ln r_i$ for each $i\in\{1,2,3\}$. The Jacobian matrix $\frac{\partial(\theta_1, \theta_2, \theta_3)}{\partial(u_1, u_2, u_3)}$ is symmetric and semi-negative definite. It has one zero eigenvalue with associated eigenvector $(1,1,1)^T$ and two negative eigenvalues.\qed
\end{lemma}

For a proof of Lemma \ref{Lemma-Guo}, see Lemma 5 and Lemma 6 in Guo \cite{Guoren}. Now we see the notion of generalized Euclidean triangles and their angles.
\begin{definition}(\cite{Bobenko,Luo1,Luo2})
A generalized Euclidean triangle $\triangle$ is a (topological) triangle of vertices $v_1, v_2, v_3$ so that each edge is assigned a positive number, called edge
length. Let $x_i$ be the assigned length of the edge $v_jv_k$ where $\{i, j, k\}$=$\{1, 2, 3\}$. The inner angle $\tilde{\theta}_i$=$\tilde{\theta}_i(x_1, x_2, x_3)$ at the vertex $v_i$ is defined as follows. If $x_1, x_2, x_3$ satisfy the triangle inequalities that $x_j+x_k>x_h$ for $\{h, j, k\}$=$\{1, 2, 3\}$, then $\tilde{\theta}_i$ is the inner angle of the Euclidean triangle of edge lengths $x_1, x_2, x_3$ opposite to the edge of length $x_i$; if $x_i\geq x_j+x_k$, then $\tilde{\theta}_i=\pi$, $\tilde{\theta}_j=\tilde{\theta}_k=0$.
\end{definition}
Thus the angle map $\theta: \Delta\rightarrow \mathds{R}^3_{>0}$ is extended to $\tilde{\theta}: \mathds{R}^3_{>0}\rightarrow \mathds{R}^3_{\geq0}$. It is known that
\begin{lemma}(Luo \cite{Luo1})
The angle function $\tilde{\theta}: \mathds{R}^3_{>0}\rightarrow \mathds{R}^3_{>0}$, $(r_1,r_2,r_3)\mapsto(\tilde{\theta}_1,\tilde{\theta}_2,\tilde{\theta}_3)$ is continuous so that $\tilde{\theta}_1+\tilde{\theta}_2+\tilde{\theta}_3=\pi$.\qed
\end{lemma}

\begin{lemma}\label{lemma-guoren-F-123}(Guo \cite{Guoren})
The smooth differential $1$-form $\omega=\theta_1 du_1+\theta_2 du_2+\theta_3 du_3$ is closed in the open subset $\ln\Delta\subset\mathds{R}^3$. For arbitrary chosen $u_0\in\ln\Delta$, the integral
\begin{equation}
F_{123}(u)\triangleq\int_{u_0}^u \theta_1 du_1+\theta_2 du_2+\theta_3 du_3, \,\,\,u\in\ln\Delta
\end{equation}
is a well defined locally concave function in $\ln\Delta$ and is strictly locally concave in $\ln\Delta\cap\{u\in \mathds{R}^3|u_1+u_2+u_3=0\}$.
\qed
\end{lemma}

Consider the whole triangulation $(M,\mathcal{T})$ with inversive distance $I\geq 0$. Remember that all vertices in $V$ are ordered one by one as $1,\cdots,N$, when we mention a triangle $\{ijk\}\in F$, we always think $i$, $j$, $k$ are naturally ordered, i.e., $i<j<k$ implicitly. For every vertex $i\in V$, the function $\tilde{\theta}_i=\tilde{\theta}_i(u)$, $u\in\mathds{R}^N$ is the continuous extension of function $\theta_i=\theta_i(u)$, $u\in\ln \Omega$. Recall the definition of discrete Gaussian curvature $K_i=2\pi-\sum_{\{ijk\} \in F}\theta_i^{jk}$, which is originally defined on $\Omega$, now can naturally extends continuously to
\begin{equation}\label{def-K-tuta}
\widetilde{K}_i=2\pi-\sum_{\{ijk\} \in F}\tilde{\theta}_i^{jk}.
\end{equation}
Hence the curvature map $K(r):\Omega\rightarrow \mathds{R}^N$ is extended continuously to
$\widetilde{K}(r):\mathds{R}^N_{>0}\rightarrow \mathds{R}^N$ with discrete Gauss-Bonnet formula remain valid.
\begin{proposition} For the extended curvature $\widetilde{K}$, the following discrete Gauss-Bonnet formula holds.
\begin{equation}\label{Gauss-Bonnet-extend}
\sum_{i=1}^N\widetilde{K}_i=2\pi \chi(M).
\end{equation}
\end{proposition}
\begin{proof}
For every generalized Euclidean triangle $\{ijk\}\in F$, $\tilde{\theta}_i^{jk}+\tilde{\theta}_j^{ik}+\tilde{\theta}_k^{ij}=\pi$. Then
\begin{equation*}
\sum_{i=1}^N\widetilde{K}_i=2\pi N-\sum_{i=1}^N\sum_{\{ijk\}\in F}\tilde{\theta}_i^{jk}
=2\pi N-\sum_{\{ijk\}\in F}\left(\tilde{\theta}_i^{jk}+\tilde{\theta}_j^{ik}+\tilde{\theta}_k^{ij}\right)
=2\pi N-\pi|F|=2\pi\chi(M).
\end{equation*}
The last equality holds, because $2|E|=3|F|$ and $\chi(M)=N-|E|+|F|$ for any triangulation of a closed surface.\qed
\end{proof}
\begin{remark}
For a triangulated surface with boundary, the curvature at a interior vertex is defined as before, while the curvature at a boundary vertex $i$ is defined as
$$\widetilde{K}_i=\pi-\sum_{\{ijk\} \in F}\tilde{\theta}_i^{jk}.$$
In this setting, the generalized Gauss-Bonnet formula remains valid. One can prove it by doubling the surface along the boundary.
\end{remark}

In the proof of Theorem \ref{Thm-Chow-Luo}, Chow-Luo studied an energy functional, which is also called ``discrete Ricci potential"
$$F(u)\triangleq\int_{u_0}^u \sum_{i=1}^N(K_i-K_{av})du_i.$$
This type of functional was first constructed by Colin de Verdi$\grave{e}$re \cite{Colindev}. $F(u)$ is well defined, since $\sum_{i=1}^N(K_i-K_{av})du_i$ is a closed smooth $1$-form and the domain where $F(u)$ is defined is simply connected. In $(M,\mathcal{T},\Phi)$ setting, this potential $F(u)$ is well defined on the whole space $u\in\mathds{R}^N$ and is convex. Furthermore, $F$ is proper if there exists a constant curvature metric. Chow-Luo's normalized flow (\ref{Def-ChowLuo's normalized flow}) can be written as $\dot{u}=-\nabla F$. This implies that $F(u(t))$ is decreasing along the flow. Assuming the existence of constant curvature metric $u_{av}$, then $u_{av}$ is the unique critical point of $F(u)$. Thus all gradient lines converge to the constant curvature metric $u_{av}$. In $(M,\mathcal{T},I)$, the inversive distance setting, $F(u)$ is well defined only on $u\in \ln\Omega$, which is not proper anymore. Although flow (\ref{Def-extended-flow-u}) is the negative gradient flow of $F(u)$, we can't derive the convergence behavior directly. A natural idea is to extend the definition of $F(u)$ with the definition domain $u\in\ln\Omega$ extended to the whole space $u\in\mathds{R}^N$, meanwhile, $F$ is still convex and proper (under the condition that there exists a constant curvature metric). Fortunately, By Feng Luo's pioneered work in \cite{Luo1}, this idea works well.

A differential $1$-form $\omega=\sum_{i=1}^na_i(x)dx_i$ in an open set $U\subset\mathds{R}^n$ is said to be continuous if each $a_i(x)$ is a continuous function on $U$. A continuous $1$-form $\omega$ is called closed if $\int_{\partial\tau}\omega=0$ for any Euclidean triangle $\tau\subset U$. By the standard approximation theory, if $\omega$ is closed and $\gamma$ is a piecewise $C^1$-smooth null homologous loop in $U$, then $\int_{\gamma}\omega=0$. If $U$ is simply connected, then in the integral
$G(x)=\int_{a}^x\omega$ is well defined (where $a\in U$ is arbitrary chose), independent of the choice of piecewise smooth paths in $U$ from $a$ to $x$. Moreover, the function $G(x)$ is $C^1$-smooth so that $\frac{\partial G(x)}{\partial x_i}=a_i(x)$.

\begin{lemma}\label{lemma-luo-essential}(Luo \cite{Luo1})
Assuming a triangle $\triangle123$ is configured by three circles with three fixed non-negative numbers $I_{12}$, $I_{23}$ and $I_{13}$ as inversive distances. The smooth $1$-form $\omega=\theta_1 du_1+\theta_2 du_2+\theta_3 du_3$ can be extended to a continuous closed $1$-form $\widetilde{\omega}=\tilde{\theta}_1 du_1+\tilde{\theta}_2 du_2+\tilde{\theta}_3 du_3$ on $\mathds{R}^3$ so that the integration
\begin{equation}\label{def-F-tuta-123}
\widetilde{F}_{123}(u)\triangleq\int_{u_0}^u \tilde{\theta}_1 du_1+\tilde{\theta}_2 du_2+\tilde{\theta}_3 du_3, \,\,\,u\in\mathds{R}^3
\end{equation}
is a $C^1$-smooth concave function.\qed
\end{lemma}

Lemma \ref{lemma-luo-essential} plays an essential role in \cite{Luo1} and \cite{Luo2}. We use this lemma to extend the definition of discrete Ricci potential $F(u)$ and derive some very useful properties of $F$. Consider the whole triangulation $(M,\mathcal{T})$ with inversive distance $I\geq 0$. Recall that $\tilde{\theta}_i=\tilde{\theta}_i(u)$, $u\in\mathds{R}^N$ is the continuous extension of function $\theta_i=\theta_i(u)$, $u\in\ln \Omega$. It's easy to see, $\tilde{\theta}_i du_i+\tilde{\theta}_j du_j+\tilde{\theta}_kdu_k$ is a continuous closed $1$-form  on $\mathds{R}^N$, hence for arbitrary chosen $u_0\in\mathds{R}^N$, the following integration
\begin{equation}
\widetilde{F}_{ijk}(u)\triangleq\int_{u_{0}}^{u}\tilde{\theta}_i du_i+\tilde{\theta}_j du_j+\tilde{\theta}_k du_k, \,\,\,u\in\mathds{R}^N
\end{equation}
is well defined and is a $C^1$-smooth concave function on $\mathds{R}^N$. Note that
\begin{equation*}
\begin{aligned}
\sum_{i=1}^N(\widetilde{K}_i-K_{av})du_i=&\sum_{i=1}^N\Big(2\pi-K_{av}-\sum_{\{ijk\}\in F}\tilde{\theta}_i^{jk}\Big)du_i\\
=&\sum_{i=1}^N(2\pi-K_{av})du_i-\sum_{i=1}^N\sum_{\{ijk\}\in F}\tilde{\theta}_i^{jk}du_i\\
=&\sum_{i=1}^N(2\pi-K_{av})du_i-\sum_{\{ijk\}\in F}\Big(\tilde{\theta}_i^{jk}du_i+\tilde{\theta}_j^{ik}du_j+\tilde{\theta}_k^{ij}du_k\Big),
\end{aligned}
\end{equation*}
which shows that $\sum_{i=1}^N(\widetilde{K}_i-K_{av})du_i$ is a continuous closed $1$-form on $\mathds{R}^N$. Therefore, we may well introduce the extended discrete Ricci potential
\begin{equation}
\widetilde{F}(u)\triangleq\int_{u_0}^u \sum_{i=1}^N(\widetilde{K}_i-K_{av})du_i, \,\,\, u\in \mathds{R}^N.
\end{equation}

\begin{proposition}
$\widetilde{F}(u)\in C^1(\mathds{R}^N)$ and is convex on the whole space $\mathds{R}^N$.
\end{proposition}
\begin{proof}
Set $u_0=(u_{0,1},\cdots,u_{0,N})^T$. It's easy to see
\begin{equation*}
\widetilde{F}(u)=C(u)-\sum_{\{ijk\}\in F}\widetilde{F}_{ijk}(u),
\end{equation*}
where $C(u)=(2\pi-K_{av})\sum_{i=1}^N(u_i-u_{0,i})$ is a linear function of $u$, while all $\widetilde{F}_{ijk}(u)$ are concave and $C^1$-smooth. Hence $\widetilde{F}(u)$ is convex. \qed
\end{proof}

\begin{proposition}\label{prop-HessF-tuta-smooth}
In $\ln\Omega$, $\widetilde{F}(u)$ is $C^{\infty}$-smooth. $Hess_u\widetilde{F}$ is positive semi-definite with rank $N-1$ and null space $\{t\mathds{1}|t\in\mathds{R}\}$.
\end{proposition}
\begin{proof}
$\ln\Omega$ is a simply connected open set. $\frac{\partial K_i}{\partial u_j}=\frac{\partial K_j}{\partial u_i}$. Hence $\omega=\sum_{i=1}^N(K_i-K_{av})du_i$ is a closed $C^{\infty}$-smooth differential $1$-form on $\ln\Omega$. For arbitrary chosen $\bar{u}\in \ln\Omega$, the line integral
\begin{equation*}
F(u)\triangleq\int_{\bar{u}}^u \sum_{i=1}^N(K_i-K_{av})du_i,
\end{equation*}
is well defined, independent of the choice of piecewise smooth paths in $\ln\Omega$ from $\bar{u}$ to $u$. Moreover, for each $u\in \ln\Omega$,
$\widetilde{F}(u)=F(u)+c$ with $c=\int_{u_0}^{\bar{u}} \sum_{i=1}^N(K_i-K_{av})du_i$ a constant. By the definition of $F$, $F$ is $C^{\infty}$-smooth. Hence in $\ln\Omega$, $\widetilde{F}(u)$ is $C^{\infty}$-smooth with $Hess_u\widetilde{F}=Hess_uF$. Denote $L=Hess_uF=\frac{\partial(K_1, \cdots, K_N)}{\partial(u_1, \cdots, u_N)}$. By Lemma \ref{Lemma-Guo}, in a single triangle $\{ijk\}\in F$, the matrix
$$L_{ijk}\triangleq-\frac{\partial(\theta_i, \theta_j, \theta_k)}{\partial(u_i, u_j, u_k)}$$
is positive semi-definite with rank $2$ and null space $\{t(1,1,1)^T|t\in\mathds{R}\}$. We extend the matrix $L_{ijk}$ to a $N\times N$ matrix. Remember that all vertices are ordered and marked by $1,\cdots,N$. We suppose $i$, $j$, $k$ arise at $i$, $j$, $k$ position respectively in the ordered sequence $1,\cdots,N$. Then we get a $N\times N$ matrix by putting $(\Lambda_{ijk})_{st}=-\frac{\partial \theta_s}{\partial u_t}$ at the $(s,t)$-entry position for any $s,t\in \{i,j,k\}$, and putting $0$ at other entries. Without confusion, we may still write the extended $N\times N$ matrix as $L_{ijk}$,
then we have $$L=\frac{\partial(K_{1},\cdots,K_{N})}{\partial(u_{1},\cdots,u_{N})}=\sum_{\{i,j,k\}\in F}L_{ijk}.$$
Because each component in the sum is positive semi-definite matrixes, $L$ is positive semi-definite .

If $Lx=0$, then $x^TLx=x^T(\sum L_{ijk})x=\sum (x^T L_{ijk}x)=0$. Hence $x^TL_{ijk}x=0$ and further $L_{ijk}x=0$ by elementary linear algebra theory, i.e., $x\in Ker(L_{ijk})$ for each $\{ijk\}\in F$.
Hence there is a constant $t_{ijk}$ so that $(x_{i},x_{j},x_{k})=t_{ijk}(1,1,1)$. Because the manifold $M$ is connected, all $t_{ijk}$ must be equal. Thus $x=t\mathds{1}$, which implies $Ker(L)=\{t\mathds{1}|t\in \mathds{R}\}$ and $rank(L)=N-1$.
 \qed
\end{proof}

\begin{proposition}\label{Prop-F-tura}
Assuming there exists a metric $u_{av}$ of constant curvature, which is unique in $\ln \Omega$ up to scaling by Theorem \ref{Thm-Guoren-Luofeng}. Denote $\mathscr{U}\triangleq \{u\in \mathds{R}^N|\sum_{i=1}^Nu_i=\sum_{i=1}^Nu_{av,i}\}$. Then $\widetilde{F}(u)$ is proper on $\mathscr{U}$ and
$\lim\limits_{u\in \mathscr{U}, u\rightarrow\infty}\widetilde{F}(u)=+\infty.$
\end{proposition}
\begin{proof}
$\widetilde{F}(u)\in C^1(\mathds{R}^N)$, and $\nabla_u\widetilde{F}=\widetilde{K}-K_{av}\mathds{1}=(\widetilde{K}_1-K_{av},\cdots,\widetilde{K}_N-K_{av})^T$. For each direction $\xi\in \mathbb{S}^{N-1}\cap\mathscr{U}$, set $\varphi_{\xi}(t)=\widetilde{F}(u_{av}+t\xi)$, $t\in \mathds{R}$. Obviously $\varphi_{\xi}\in C^1(\mathds{R})$, and $\varphi_{\xi}'(t)=(\widetilde{K}-K_{av}\mathds{1})\cdot\xi$. $\varphi_{\xi}(t)$ is convex on $\mathds{R}$ since $\widetilde{F}$ is convex on $\mathds{R}^N$, hence $\varphi'_{\xi}(t)$ is increasing on $\mathds{R}$. Note $u_{av}\in \ln \Omega$, hence there exists $c>0$ so that for each $t\in[-c, c]$, $u_{av}+t\xi$ remains in $\ln \Omega$. Note $\widetilde{F}(u)$ is $C^{\infty}$-smooth on $\ln\Omega$, hence $\varphi_{\xi}(t)$ is $C^{\infty}$-smooth for $t\in[-c, c]$.
Note that the kernel space of $Hess_u\widetilde{F}$ is $\{t\mathds{1}|t\in\mathds{R}\}$, which is perpendicular to the hyperplane $\mathscr{U}$. Hence $\widetilde{F}(u)$ is strictly convex locally in $\ln\Omega\cap\mathscr{U}$. Then $\varphi_{\xi}(t)$ is strictly convex at least on a small interval $[-\delta, \delta]$. This implies that $\varphi'_{\xi}(t)$ is a strictly increase function on $[-\delta, \delta]$. Note that $\varphi'_{\xi}(0)=0$, hence $\varphi'_{\xi}(t)\geq\varphi'_{\xi}(\delta)>0$ for $t>\delta$ while $\varphi'_{\xi}(t)\leq\varphi'_{\xi}(-\delta)<0$ for $t<-\delta$. Hence
$$\varphi_{\xi}(t)\geq\varphi_{\xi}(\delta)+\varphi'_{\xi}(\delta)(t-\delta)$$
for $t\geq \delta$, while
$$\varphi_{\xi}(t)\geq\varphi_{\xi}(-\delta)+\varphi'_{\xi}(-\delta)(t+\delta)$$
for $t\leq -\delta$. This implies
\begin{equation}
\lim\limits_{t\rightarrow\pm\infty}\varphi_{\xi}(t)=+\infty.
\end{equation}

One may use the following Lemma \ref{lemma-gotoinfinity} to get the conclusion.
 \qed
\end{proof}

\begin{lemma}\label{lemma-gotoinfinity}
Assuming $f\in C(\mathds{R}^n)$ and for any direction $\xi\in \mathbb{S}^{n-1}$, $f(t\xi)$ as a function of $t$ is monotone increasing on $[0,+\infty)$ and tends to $+\infty$ as $t\rightarrow +\infty$. Then $\lim\limits_{x\rightarrow\infty}f(x)=+\infty.$
\end{lemma}
\begin{proof}
We follow the proof of Lemma B.1 in \cite{Ge-Xu1}. Denote $h(t)\triangleq \inf\limits_{\|x\|=t} f(x)$. We just need to prove $\lim\limits_{t\rightarrow+\infty}h(t)=+\infty$. If not, since $h(t)$ is a monotone increasing function, we may find $M>0$, such that $h(t)<M$ for all $t\geq0$. For each $k\in\ \mathds{N}$, choose $x_k \in\mathds{R}^n$ such that $\|x_k\|=k$ and $f(x_k)<M$.
Denote $\overset{\circ}{x_k}=\frac{x_k}{\|x_k\|}$.
$\{\overset{\circ}{x_k}\}$ must has a convergent subsequence, denoted as $\{\overset{\circ}{x_k}\}$ too.
Suppose $\overset{\circ}{x_k}\rightarrow x^*\in \mathbb{S}^{n-1}$.
Then there exists a positive integer $a>0$, such that $f(tx^*)>M$ for all $t\geq a$. Let $\mathbb{S}^{n-1}(a)\triangleq \{x\in\mathds{R}^n|\|x\|=a\}$. Select $\mu>0$ so that for all $x\in B(ax^*,\mu)\cap \mathbb{S}^{n-1}(a)$ we have $f(x)>M$. However, $B(x^*,\frac{\mu}{a})\cap \mathbb{S}^{n-1}$ is an open neighborhood (relative to $\mathbb{S}^{n-1}$) of $x^*$. So there exists at least one $m>a$ such that $\overset{\circ}{x_m}\in B(x^*,\frac{\mu}{a})\cap \mathbb{S}^{n-1}$. Then $a\overset{\circ}{x_m}\in B(ax^*,\mu)\cap \mathbb{S}^{n-1}(a)$, hence
$$f(x_m)=f(m\overset{\circ}{x_m})\geq f(a\overset{\circ}{x_m})>M,$$
which contradicts the selection of $x_k$. Thus we get the conclusion above.\qed
\end{proof}

\begin{corollary}\label{corollary-unique-K-average}
Assuming there exists a metric $u_{av}\in \ln \Omega$ of constant curvature (this metric is unique up to scaling). Then it is unique in the extended space $\mathds{R}^N$ up to scaling.
\end{corollary}
\begin{proof}
Suppose $u^*\in \mathscr{U}$ is also a metric of constant curvature, with $u^*\neq u_{av}$. Write $t_0=\|u^*-u_{av}\|$ and  $\xi=t_0^{-1}(u^*-u_{av})$. Then $\varphi_{\xi}'(t_0)=(\widetilde{K}(u^*)-K_{av}\mathds{1})\cdot\xi=0$. On the other hand, we had proved in Proposition \ref{Prop-F-tura} that $t=0$ is the unique zero point of $\varphi_{\xi}'(t)$, which is a contradiction. \qed
\end{proof}

\begin{theorem}\label{thm-section3-mainthm}
The solution to the extended flow
\begin{equation}
\begin{cases}
{u_i}'(t)=K_{av}-\widetilde{K}_i\\
\;\,u(0)\in \ln\Omega
\end{cases}
\end{equation}
exists for all time $t\in[0,+\infty)$. If further assuming there exists a metric $u_{av}\in \ln \Omega$ of constant curvature, then the solution converges to $u_{av}$ exponentially fast.
\end{theorem}
\begin{proof}
All $\widetilde{K}_i$ are uniformly bounded, hence the solution to the extended flow (\ref{Def-extended-flow-u}) exists for all time. Next we prove the convergence part.

Flow (\ref{Def-extended-flow-u}) is a negative gradient flow; that is we may write it as $\dot{u}=-\nabla \widetilde{F}$.
Let $\varphi(t)=\widetilde{F}(u(t))$, then $\varphi'(t)=-\|\widetilde{K}-K_{av}\mathds{1}\|^2\leq 0$. Hence $\varphi(t)$ is decreasing. Assuming the existence of constant curvature metric $u_{av}$, $\widetilde{F}$ is proper, hence $\varphi(t)$ is compactly supported in $\mathds{R}^N$. Also, $\widetilde{F}$ is bounded form below, hence $\varphi(+\infty)$ exists. By the mean value theorem, there exists a sequence $t_n\in(n,n+1)$ such that $\varphi'(t_n)=\varphi(n+1)-\varphi(n)\rightarrow 0$. By choosing a subsequence of $t_n$, which is still denote as $t_n$, we require that $u(t_n)$ converges to some point $u^*$. Hence $\widetilde{K}(u(t_n))$ converges to $\widetilde{K}(u^*)$. Using $\varphi'(t_n)=-\|\widetilde{K}(u(t_n))-K_{av}\mathds{1}\|^2\rightarrow 0$, we have $\widetilde{K}(u^*)=K_{av}\mathds{1}$. By the uniqueness of constant curvature metric in Corollary \ref{corollary-unique-K-average}, we get $u^*=u_{av}$ and $u(t_n)\rightarrow u_{av}$. Note in $\ln \Omega$, $\widetilde{K}=K$ is differentiable, then differentiate the right hand side of the extended flow at $u_{av}$, we get
$$D_u(K_{av}\mathds{1}-\widetilde{K})\big|_{u_{av}}=-D_uK=-\frac{\partial(K_{1},\cdots,K_{N})}{\partial(u_{1},\cdots,u_{N})}=-L.$$
Remember that $L$ is negative definite along the flow, this implies that $u_{av}$ is the asymptotically stable point of the extended flow. Note that for some sufficient big $t_{n}$, $u(t_{n})$ is
very close to constant curvature metric $u_{av}$. Then the solution $\{u(t)\}_{t\geq t_n}$ converges exponentially fast to $u_{av}$, i.e., the original solution $\{u(t)\}_{t\geq 0}$ converges exponentially fast to $u_{av}$.
\qed
\end{proof}

\begin{remark}
Theorem \ref{thm-section3-mainthm} is still true if we change $u(0)\in \ln \Omega$ to arbitrary initial value, i.e., to $u(0)\in \mathds{R}^N$. In fact, the proof above is irrelevant with the selection of $u(0)$. Thus we can deform the inversive distance metric to constant curvature metric from any initial value $r(0)\in \mathds{R}^N_{>0}$, even if $r(0)$ is not a real inversive distance metric. It means that, for practical applications we don't need to verify the triangle inequalities of $r(0)$ any more.
\end{remark}

\section{Deform the metric to prescribed curvature}\label{section-andreev-thurston}
For any prescribed value $\bar{K}=(\bar{K}_1,\cdots,\bar{K}_N)^T\in\mathds{R}^N$, differential form $\sum_{i=1}^N(\widetilde{K}_i-\bar{K}_{i})du_i$ is a continuous closed $1$-form on $\mathds{R}^N$. Hence the prescribed discrete Ricci potential
\begin{equation}\label{def-G-tuta}
\widetilde{G}(u)\triangleq\int_{u_0}^u \sum_{i=1}^N(\widetilde{K}_i-\bar{K}_{i})du_i, \,\,\, u\in \mathds{R}^N
\end{equation}
is well defined for any $u_0\in \mathds{R}^N$. It is the $C^1$-smooth extension of
\begin{equation}
G(u)\triangleq\int_{u'_0}^u \sum_{i=1}^N(K_i-\bar{K}_{i})du_i+C, \,\,\, u\in\ln\Omega,
\end{equation}
where $u'_0\in \ln\Omega$ is arbitrary chosen and the constant $C=\int_{u_0}^{u'_0} \sum_{i=1}^N(\widetilde{K}_i-\bar{K}_{i})du_i$. It's easy to see
\begin{equation*}
\widetilde{G}(u)=\sum_{i=1}^N(2\pi-K_{av})(u_i-\bar{u}_{i})-\sum_{\{ijk\}\in F}\widetilde{F}_{ijk}(u),
\end{equation*}
hence $\widetilde{G}(u)\in C^1(\mathds{R}^N)$ and is convex on the whole space $\mathds{R}^N$. Similar to Proposition \ref{prop-HessF-tuta-smooth}, we can prove that in $\ln\Omega$, $\widetilde{G}(u)$ is $C^{\infty}$-smooth. $Hess_u\widetilde{G}=Hess_u\widetilde{F}$ is positive semi-definite with rank $N-1$ and null space $\{t\mathds{1}|t\in\mathds{R}\}$. If further assuming $\sum_{i=1}^N\bar{K}_i=2\pi\chi(M)$, then
$$\widetilde{G}(u)=\widetilde{G}(u+t\mathds{1})$$
for any $t\in\mathds{R}$ and any $u\in\mathds{R}^N$.

\begin{definition}
Given the triangulation $(M,\mathcal{T})$ with inversive distance $I\geq 0$, $\Omega$ is the space of all possible inversive distance circle packing metrics, consider $K$ as the function of $r$ and denote $K(\Omega)\triangleq\big\{K(r)\big|r\in \Omega\big\}$. Each prescribed $\bar{K}$ with $\bar{K}\in K(\Omega)$ is called admissible. If $\bar{r}\in\Omega$ such that $\bar{K}=K(\bar{r})$, we say $\bar{K}$ is realized by $\bar{r}$.
\end{definition}
\begin{remark}
Note that the $\bar{r}$ is unique (up to a scalar multiplication) in $\Omega$ that realizes $\bar{K}$ by Theorem \ref{Thm-Guoren-Luofeng} proved in Guo \cite{Guoren} and Luo \cite{Luo1}.
\end{remark}

For any admissible prescribed curvature $\bar{K}$ that is realized by $\bar{r}\in\Omega$, let $\bar{u}\in\ln\Omega$ be the corresponding metric in $u$-coordinate and $\widetilde{G}(u)$ is defined as in formula (\ref{def-G-tuta}). Denote
$$\bar{\mathscr{U}}\triangleq \{u\in \mathds{R}^N|\sum_{i=1}^Nu_i=\sum_{i=1}^N\bar{u}_{i}\}.$$ Similar to Proposition \ref{Prop-F-tura}, we can prove $\widetilde{G}(u)$ is proper on $\bar{\mathscr{U}}$ and $\lim\limits_{u\in \bar{\mathscr{U}}, u\rightarrow\infty}\widetilde{G}(u)=+\infty$. Using this fact, we can extend Guo-Luo's global rigidity results, i.e., Theorem \ref{Thm-Guoren-Luofeng} to the following
\begin{theorem}\label{corollary-unique-K-bar}
Assuming the curvature $\bar{K}\in K(\Omega)$ is admissible. Then it is realized by an unique metric $\bar{r}$ in the extended space $\mathds{R}^N_{>0}$ up to a scalar multiplication.\qed
\end{theorem}
Moreover, using the following prescribed curvature flow, we can deform any extended inversive distance circle packing metric $r(0)\in\mathds{R}^N_{>0}$ to any metric with admissible prescribed curvature.
\begin{theorem}\label{thm-section4-mainthm}
Consider the prescribed curvature flow
\begin{equation}\label{def-extended-flow-u-prescribed}
\begin{cases}
{u_i}'(t)=\bar{K}_{i}-\widetilde{K}_i\\
\;\,u(0)\in\mathds{R}^N
\end{cases}
\end{equation}
\begin{description}
  \item[(1)] The solution $u(t)$ to flow (\ref{def-extended-flow-u-prescribed}) exists for all time $t\in[0,+\infty)$.
  \item[(2)] If $u(t)$ converges to some $\bar{u}$, then $\bar{K}$ is admissible and is realized by $\bar{u}$.
  \item[(3)] If $\bar{K}$ is admissible and is realized by $\bar{u}$, then $u(t)$ converges to $\bar{u}$ exponentially fast.
\end{description}\qed
\end{theorem}

\begin{corollary}
Given a triangulated surface $(M, \mathcal{T}, I)$ with inversive distance $I\geq 0$. Consider the following prescribed curvature flow
\begin{equation}
\begin{cases}
{r_i}'(t)=(\bar{K}_{i}-K_i)r_i\\
\;\,r(0)\in \Omega.
\end{cases}
\end{equation}
Suppose $\{r(t)|t\in[0,T)\}$ is the unique maximal solution with $0<T\leq +\infty$. Then we can always extend it to a solution $\{r(t)|t\in[0,+\infty)\}$ when $T<+\infty$. Furthermore, if $\bar{K}$ is admissible and is realized by $\bar{r}$, then $r(t)$ can always be extended to a solution that converges exponentially fast to $\bar{r}$ as $t\rightarrow+\infty$.\qed
\end{corollary}

Since all the proofs are similar with last section, we omit the details here.

\section{Partial results for $K(\Omega)$}\label{section-andreev-thurston}
As we have seen in Theorem \ref{thm-section4-mainthm}, convergence of $r(t)$ is equivalent to $\bar{K}$ admissible. We want to describe the shape of
$K(\Omega)$. For $(M, \mathcal{T}, \Phi)$ with circle packing metric setting, $\Omega$ equals to $\mathds{R}^N_{>0}$, the classical Andreev-Thurston's theorem completely describes the shape of $K(\mathds{R}^N_{>0})$. Fix a triangulated surface $(M, \mathcal{T})$. For any nonempty proper subset $A\subset V$, let $F_A$ be the subcomplex whose vertices are in $A$ and let $Lk(A)$ be the set of pairs $(e, v)$ of an edge $e$ and a vertex $v$ satisfying the following three conditioins: (1) The end points of $e$ are not in $A$; (2) $v$ is in $A$; (3) $e$ and $v$ form a triangle.

First we consider circle packing case (with weight $\Phi\in[0,\frac{\pi}{2}]$). In this setting, the admissible metric space is $\mathds{R}^N_{>0}$.
Denote
\begin{equation}\label{Y_I half space}
Y_A^{\Phi}\triangleq\Big\{x\in \mathds{R}^N \Big|\sum_{i\in A}x_i >-\sum_{(e,v)\in Lk(A)}\big(\pi-\Phi(e)\big)+2\pi\chi(F_A)\Big\}.
\end{equation}
Obviously, the discrete Gaussian curvature $K$ is uniquely determined by circle packing metric $r$. Hence $K$=$K(r):\mathds{R}^N_{>0}\rightarrow \mathds{R}^N$ is a vector-valued function of $r$, which is called the curvature map. In the following, the symbol $K$ either means a concrete discrete Gaussian curvature or the curvature map, which is clear from the context. The classical Andreev-Thurston's theorem says that the curvature map $K=K(r)$ is injective (up to scaling), and
the space of all possible curvatures $K(\mathds{R}^N_{>0})\triangleq\big\{K(r)\big|r\in \mathds{R}^N_{>0}\big\}$ is completely determined by $Y_A^{\Phi}$.

\begin{theorem}\label{Thm-Andreev-thurston} \textbf{(Andreev-Thurston)}
\;Given a weighted triangulated surface $(M, \mathcal{T}, \Phi)$, the curvature map $K$ restricted to the subset $\big\{r\in \mathds{R}^N_{>0}\big|\prod_{i=1}^N=1\big\}$ is injective, i.e., the metric is determined by its curvature up to a scalar multiplication. Moreover,
the space of all possible discrete Gaussian curvatures $K(\mathds{R}^N_{>0})$ equals to the following convex set
\begin{equation}\label{K-space}
\Big\{x\in \mathds{R}^N\Big|\sum_{i=1}^Nx_i=2\pi\chi(M)\Big\}\bigcap\Big(\mathop{\bigcap} _{\phi\neq A\subsetneqq V} Y_A^{\Phi} \Big).
\end{equation}
\end{theorem}

For a proof, see Thurston \cite{T1}, Marden-Rodin \cite{Marden-Rodin}, Colin de Verdi$\grave{e}$re \cite{Colindev}, He \cite{Hezhengxu1}, Chow-Luo \cite{CL1} and Stephenson \cite{Stephenson}. Note that the existence of constant curvature metric is equivalent to $K_{av}\mathds{1}\in K(\mathds{R}^N_{>0})$. Substitute $K_{av}\mathds{1}$ into (\ref{K-space}) and (\ref{Y_I half space}), one can easily get:

\begin{theorem}\label{Thm-thurston} \textbf{(Thurston)} \;Given a weighted triangulated surface $(M, \mathcal{T}, \Phi)$,
the existence of constant curvature metric is equivalent to the following combinatorial and topological conditions
\begin{equation}\label{condition-Thurston}
2\pi\chi(M)\frac{|A|}{|V|} >-\sum_{(e,v)\in Lk(A)}\big(\pi-\Phi(e)\big)+2\pi\chi(F_A), \;\;\forall A: \phi\subsetneqq A\subsetneqq V.
\end{equation}
Moreover, if constant curvature metric exists, it is unique up to a scalar multiplication.
\end{theorem}

In the following we consider the inversive distance circle packing case (with inversive distance $I\geq 0$). In this setting, the admissible
metric space is $\Omega$. Based on extensive numerical evidences, Bowers and Stephenson conjectured the rigidity of inversive distance circle packing in \cite{Bowers-Stephenson}. Guo \cite{Guoren} first proved that Bowers-Stephenson's conjecture of rigidity is locally true by complicated computations. Luo \cite{Luo1} proved Bowers-Stephenson's conjecture of rigidity eventually.

\begin{theorem}\label{Thm-Guoren-Luofeng} \textbf{(Guo-Luo)}
\;Given a triangulated surface $(M, \mathcal{T}, I)$ with inversive distance $I\geq 0$. The curvature map $K:\Omega\rightarrow\mathds{R}^N$ restricted to the subset $\big\{r\in\Omega\big|\prod_{i=1}^Nr_i=1\big\}$ is injective, i.e., the metric is determined by its curvature up to a scalar multiplication.
\end{theorem}

Theorem \ref{Thm-Guoren-Luofeng} confirms Bowers and Stephenson's rigidity conjecture. It's a generalization of the uniqueness part of Andreev-Thurston's Theorem \ref{Thm-Andreev-thurston} to the inversive distance setting. However, the generalization of the existence part still remains unresolved. In the following of this section, we will give a partial answer to the existence part of Andreev-Thurston theorem in Bowers-Stephenson's inversive distance setting. We follow the approach pioneered by Marden and Rodin \cite{Marden-Rodin}. Denote
\begin{equation}\label{def-Y-A}
Y_A\triangleq\Big\{x\in \mathds{R}^N \Big|\sum_{i\in A}x_i >-\sum_{(e,v)\in Lk(A)}\big(\pi-\Lambda(I_e)\big)+2\pi\chi(F_A)\Big\},
\end{equation}
for any nonempty proper subset $A\subset V$, where
\begin{equation}
\Lambda(x)=
\begin{cases}
\;\;\;\;\;\pi, & \text{$x \leq -1$.}\\
\arccos x, & \text{$-1\leq x \leq 1$.}\\
\;\;\;\;\;0,& \text{$x\geq 1$.}
\end{cases}
\end{equation}
Note that $\Lambda$ is continuous on $\mathds{R}$ and $\Lambda(-x)=\pi-\Lambda(x)$ for each $x\in \mathds{R}$. We have the following main theorem in this section:

\begin{theorem}\label{Thm-2}
Given a triangulated surface $(M, \mathcal{T}, I)$ with inversive distance $I\geq 0$. Then the space of all possible discrete Gaussian curvatures $K(\Omega)$ is contained in the following convex set
\begin{equation}\label{convex-set}
\Big\{x\in \mathds{R}^N\Big|\sum_{i\in V}x_i=2\pi\chi(M)\Big\}\bigcap\Big(\mathop{\bigcap} _{\phi\neq A\subsetneqq V} Y_A \Big).
\end{equation}
\end{theorem}

\begin{proof}
We just need to prove
\begin{equation*}
\sum_{i\in A}K_i >-\sum_{(e,v)\in Lk(A)}\big(\pi-\Lambda(I_e)\big)+2\pi\chi(F_A).
\end{equation*}
First, we prove
\begin{claim}\label{claim-1}
For each $(r_1,r_2,r_3)^T\in\Delta$ and $\{i,j,k\}=\{1,2,3\}$, $0<\theta_i<\pi-\Lambda(I_{jk})$.
\end{claim}
For this, we just need to prove $\theta_i(r_i, \bar{r}_j, \bar{r}_k)<\pi-\Lambda(I_{jk})$ for any fixed $\bar{r}_j, \bar{r}_k\in (0,+\infty)$. If $I_{jk}\in[0,1]$, set $\bar{r}_i=0$. If $I_{jk}>1$, let $\bar{r}_i>0$ be the unique solution of equation
\begin{equation}\label{jk-i-cone-surface}
\sqrt{r_i^2+\bar{r}_j^2+2r_i\bar{r}_jI_{ij}}+\sqrt{r_i^2+\bar{r}_k^2+2r_i\bar{r}_kI_{ik}}=\sqrt{\bar{r}_j^2+\bar{r}_k^2+2\bar{r}_j\bar{r}_kI_{jk}}.
\end{equation}
That equation (\ref{jk-i-cone-surface}) has a unique positive solution $\bar{r}_i$ can be seen as follows. Let
\begin{equation*}
f(r_i)=\sqrt{r_i^2+\bar{r}_j^2+2r_i\bar{r}_jI_{ij}}+\sqrt{r_i^2+\bar{r}_k^2+2r_i\bar{r}_kI_{ik}}-\sqrt{\bar{r}_j^2+\bar{r}_k^2+2\bar{r}_j\bar{r}_kI_{jk}}.
\end{equation*}
It's easy to see $f(0)<0$, $f(+\infty)=+\infty$ and $f'(r_i)>0$. Then equation (\ref{jk-i-cone-surface}) has a unique positive solution $\bar{r}_i$.

On one hand, by the law of cosines,
$$\cos \theta_i=\cfrac{l_{ij}^2+l_{ik}^2-l_{jk}^2}{2l_{ij}l_{ik}},$$
and taking limit, we get
\begin{equation}\label{xita-i-1}
\lim_{r_i\rightarrow \bar{r}_i}\theta_i(r_i, \bar{r}_j, \bar{r}_k)=\pi-\Lambda(I_{jk}).
\end{equation}
On the other hand, by Lemma \ref{Lemma-Guo}, $\frac{\partial \theta_i}{\partial r_i}r_i=\frac{\partial \theta_i}{\partial u_i}\leq 0$, which implies that $\theta_i$ is a decreasing function of $r_i$. Hence
$$\theta_i(r_i, \bar{r}_j, \bar{r}_k)\leq\pi-\Lambda(I_{jk}).$$

Next we show the equality can never be achieved. If $\theta_i(a, \bar{r}_j, \bar{r}_k)=\pi-\Lambda(I_{jk})$ at some point $a$ with $a>\bar{r}_i$ and $(a, \bar{r}_j, \bar{r}_k)\in \Delta_r$, then $\theta_i(r_i, \bar{r}_j, \bar{r}_k)\equiv\pi-\Lambda(I_{jk})$ on interval $(\bar{r}_i, a]$. By the law of sines,
$$\cfrac{l_{jk}}{\sin \theta_i}=\cfrac{l_{ij}}{\sin \theta_k}=\cfrac{l_{ik}}{\sin \theta_j}.$$
As $r_i$ increases in the interval $(\bar{r}_i, a]$, $l_{ij}$ and $l_{ik}$ increase and hence $\sin \theta_k$ and $\sin \theta_j$ increase.
Note that $\theta_i\equiv\pi-\Lambda(I_{jk})\geq\frac{\pi}{2}$, which implies that both $\theta_k$ and $\theta_j$ are in $(0, \frac{\pi}{2})$. Therefore both $\theta_k$ and $\theta_j$ are increase. However, $\theta_i+\theta_j+\theta_k=\pi$. Now we get a contradiction, which means that $\theta_i(a, \bar{r}_j, \bar{r}_k)$ never equals to $\pi-\Lambda(I_{jk})$. Hence $0<\theta_i<\pi-\Lambda(I_{jk})$.

Now we begin the proof of the theorem. Consider all the triangles in $F$ having a vertex in $A$. These triangles can be classified into three types $A_1$, $A_2$ and $A_3$. For each $i\in\{1,2,3\}$, a triangle is in $A_i$ if and only if it has exactly $i$ many vertices in $A$. Set $a_i$ as the cone angle at vertex $i$, i.e.,
$a_i=\sum_{\{ijk\} \in F}\theta_i^{jk}$. On one hand, by Claim \ref{claim-1},
\begin{equation*}\label{A-1-inequal}
\sum_{i\in A, \{ijk\}\in A_1} \theta_i^{jk}<\sum_{(e,v)\in Lk(A)}\big(\pi-\Lambda(I_e)\big).
\end{equation*}
On the other hand,
\begin{equation*}\label{A-2-inequal}
\sum_{i,j\in A, \{ijk\}\in A_2} (\theta_i^{jk}+\theta_j^{ik})<\pi|A_2|.
\end{equation*}
Note that $A_1$, $A_2$ can't be empty at the same time, hence
\begin{equation*}
\begin{aligned}
\sum_{i\in A}K_i=&\sum_{i\in A}(2\pi-a_i)\\
=&2\pi|A|-\sum_{i\in A}a_i\\
=&2\pi|A|-\Big(\sum_{i\in A, \{ijk\}\in A_1} \theta_i^{jk}+\sum_{i,\;j\in A, \{ijk\}\in A_2} (\theta_i^{jk}+\theta_j^{ik})+
\sum_{\{ijk\}\in A_3} (\theta_i^{jk}+\theta_j^{ik}+\theta_k^{ij})\Big)\\[6pt]
>&2\pi|A|-\sum_{(e,v)\in Lk(A)}\big(\pi-\Lambda(I_e)\big)-\pi|A_2|-\pi|A_3|\\
=&-\sum_{(e,v)\in Lk(A)}\big(\pi-\Lambda(I_e)\big)+2\pi\Big(|A|-\frac{|A_2|}{2}-\frac{|A_3|}{2}\Big)\\
=&-\sum_{(e,v)\in Lk(A)}\big(\pi-\Lambda(I_e)\big)+2\pi\chi(F_A).
\end{aligned}
\end{equation*}\qed
\end{proof}

\begin{corollary}\label{crollary}
Given a triangulated surface $(M, \mathcal{T}, I)$ with inversive distance $I\geq 0$. Assuming there exists a metric of constant curvature, then the following combinatorial-topological conditions holds for each nonempty proper subset $A\subset V$,
\begin{equation}\label{condition-constant-generalized}
2\pi\chi(M)\frac{|A|}{|V|} >-\sum_{(e,v)\in Lk(A)}\big(\pi-\Lambda(I_e)\big)+2\pi\chi(F_A).
\end{equation}\qed
\end{corollary}

\begin{lemma}\label{Lemma-limit-xita}
Assuming $b, c\in(0, +\infty]$, then in the generalized Euclidean triangle $\triangle 123$,
\begin{equation}
\lim_{(r_i, r_j, r_k)\rightarrow (0,\,b,\,c)}\tilde{\theta}_i(r_i, r_j, r_k)=\pi-\Lambda(I_{jk}).
\end{equation}
\begin{equation}
\lim_{(r_i,r_j,r_k)\rightarrow (0,\,0,\,c)}\tilde{\theta}_k(r_i, r_j, r_k)=0.
\end{equation}
where $\{i,j,k\}=\{1,2,3\}$.
\end{lemma}
\begin{proof} For any $(r_1,r_2,r_3)\in\mathds{R}^3_{>0}$, we have $l_{12}, l_{13}, l_{23}>0$. Now let $\{i,j,k\}=\{1,2,3\}$. If $l_{ij}, l_{ik}, l_{jk}$ satisfy the triangle inequality, which is equivalent to $-1<\frac{l_{ik}^2+l_{ij}^2-l_{jk}^2}{2l_{ik}l_{ij}}<1$, then $\tilde{\theta}_i=\theta_i$. If $l_{jk}\geq l_{ij}+l_{ik}$, which is equivalent to $\frac{l_{ik}^2+l_{ij}^2-l_{jk}^2}{2l_{ik}l_{ij}}\leq -1$, then $\tilde{\theta}_i=\pi$. Else the only left case is $l_{jk}\leq |l_{ij}-l_{ik}|$, which is equivalent to $\frac{l_{ik}^2+l_{ij}^2-l_{jk}^2}{2l_{ik}l_{ij}}\geq 1$, then $\tilde{\theta}_i=0$. Above all, we get
\begin{equation}
\tilde{\theta}_i(r_i, r_j, r_k)=\Lambda\bigg(\frac{l_{ik}^2+l_{ij}^2-l_{jk}^2}{2l_{ik}l_{ij}}\bigg).
\end{equation}
Hence
\begin{equation*}
\begin{aligned}
\tilde{\theta}_i(r_i, r_j, r_k)=&\Lambda\bigg(\frac{l_{ik}^2+l_{ij}^2-l_{jk}^2}{2l_{ik}l_{ij}}\bigg)\\
                               =&\Lambda\left(\frac{r_{i}^2+r_{i}r_{k}I_{ik}+r_{i}r_{j}I_{ij}-r_{j}r_{k}I_{jk}}{\sqrt{r_{i}^2+r_{k}^2+2r_{i}r_{k}I_{ik}}
                               \sqrt{r_{i}^2+r_{j}^2+2r_{i}r_{j}I_{ij}}}\right)\rightarrow & \Lambda(-I_{jk})=\pi-\Lambda(I_{jk})
\end{aligned}
\end{equation*}
as $(r_i, r_j, r_k)\rightarrow (0,\,b,\,c)$, while
\begin{equation*}
\begin{aligned}
\tilde{\theta}_k(r_i, r_j, r_k)=&\Lambda\bigg(\frac{l_{ik}^2+l_{jk}^2-l_{ij}^2}{2l_{ik}l_{jk}}\bigg)\\
                               =&\Lambda\left(\frac{r_{k}^2+r_{i}r_{k}I_{ik}+r_{j}r_{k}I_{jk}-r_{i}r_{j}I_{ij}}{\sqrt{r_{i}^2+r_{k}^2+2r_{i}r_{k}I_{ik}}
                               \sqrt{r_{j}^2+r_{k}^2+2r_{j}r_{k}I_{jk}}}\right)\rightarrow & \Lambda(1)=0
\end{aligned}
\end{equation*}
as $(r_i, r_j, r_k)\rightarrow (0,\,0,\,c)$.\qed
\end{proof}

\begin{proposition}\label{degenerate}
Given a triangulated surface $(M, \mathcal{T}, I)$ with inversive distance $I\geq 0$. Assuming there is a sequence of
$r^{(n)}=\big(r_1^{(n)},...,r_N^{(n)}\big)^T\in \mathds{R}^N_{>0}$ and a nonempty proper subset $A\subset V$, so that
$\lim\limits_{n\rightarrow+\infty}r_i^{(n)}=0$ for $i\in A$ and $\lim\limits_{n\rightarrow+\infty}r_i^{(n)}>0$ (may be $+\infty$) for $i\notin A$, then
\begin{equation}\label{limit-singular-behavior}
\lim\limits_{n\rightarrow+\infty}\sum_{i\in \,A}\widetilde{K}_i(r^{(n)})=-\sum_{(e,v)\in Lk(A)}(\pi-\Lambda(I_e))+2\pi\chi(F_A).
\end{equation}
\end{proposition}
\begin{proof}
The proof is similar with Theorem \ref{Thm-2}. Note that for generalized metric $r^{(n)}\in \mathds{R}^N_{>0}$, the topological triangles in $F$ may not Euclidean, it may be a generalized Euclidean triangle. However, the combinatorial structure is invariant since we fixed the triangulation. Consider all the topological triangles in $F$ having a vertex in $A$. These triangles can be classified into three types $A_1$, $A_2$ and $A_3$. For each $i\in\{1,2,3\}$, a triangle is in $A_i$ if and only if it has exactly $i$ many vertices in $A$. Let $\tilde{a}_i$ be the generalized cone angle at vertex $i$, i.e.,
$\tilde{a}_i=\sum_{\{ijk\} \in F}\tilde{\theta}_i^{jk}$. On one hand, by Lemma \ref{Lemma-limit-xita},
\begin{equation*}\label{A-1-inequal}
\sum_{i\in A, \{ijk\}\in A_1} \tilde{\theta}_i^{jk(n)}\rightarrow\sum_{(e,v)\in Lk(A)}\big(\pi-\Lambda(I_e)\big).
\end{equation*}
On the other hand,
\begin{equation*}\label{A-2-inequal}
\sum_{i,j\in A, \{ijk\}\in A_2} (\tilde{\theta}_i^{jk(n)}+\tilde{\theta}_j^{ik(n)})=\sum_{i,j\in A, \{ijk\}\in A_2} (\pi-\tilde{\theta}_k^{ij(n)})\rightarrow|A_2|\pi.
\end{equation*}
hence
\begin{equation*}
\begin{aligned}
\sum_{i\in A}\widetilde{K}_i^{(n)}
=&\sum_{i\in A}(2\pi-\tilde{a}_i^{(n)})\\
=&2\pi|A|-\sum_{i\in A}\tilde{a}_i^{(n)}=2\pi|A|-\\[2pt]
&\Big(\sum_{i\in A, \{ijk\}\in A_1} \tilde{\theta}_i^{jk(n)}+\sum_{i,\;j\in A, \{ijk\}\in A_2} (\tilde{\theta}_i^{jk}+\tilde{\theta}_j^{ik})^{(n)}+
\sum_{\{ijk\}\in A_3} (\tilde{\theta}_i^{jk}+\tilde{\theta}_j^{ik}+\tilde{\theta}_k^{ij})^{(n)}\Big)\\[6pt]
\rightarrow&2\pi|A|-\sum_{(e,v)\in Lk(A)}\big(\pi-\Lambda(I_e)\big)-|A_2|\pi-|A_3|\pi\\
=&-\sum_{(e,v)\in Lk(A)}\big(\pi-\Lambda(I_e)\big)+2\pi\chi(F_A).
\end{aligned}
\end{equation*}\qed
\end{proof}

\begin{theorem}\label{com-inequi-1}
Given a triangulated surface $(M, \mathcal{T}, I)$ with inversive distance $I\geq 0$. Then the space of all possible extended curvatures $\widetilde{K}(\mathds{R}^N_{>0})$ is contained in the closer of convex set (\ref{convex-set}).
\end{theorem}
\begin{proof}
We need to prove that for every $r\in \mathds{R}^N_{>0}$, the extended curvature $\widetilde{K}$ satisfies
\begin{equation*}
\sum_{i\in \,A}\widetilde{K}_i(r)\geq-\sum_{(e,v)\in Lk(A)}(\pi-\Lambda(I_e))+2\pi\chi(F_A)
\end{equation*}
for each nonempty proper subset $A\subset V$. First we prove
\begin{claim}\label{claim-2}
For each $(r_1,r_2,r_3)^T\in\mathds{R}^3_{>0}$ and $\{i,j,k\}=\{1,2,3\}$, $0\leq\tilde{\theta}_i\leq\pi-\Lambda(I_{jk})$.
\end{claim}
In fact, for $I_{jk}\geq 1$ case, then obviously $0\leq\tilde{\theta}_i\leq\pi=\pi-\Lambda(I_{jk})$. Else for $0\leq I_{jk}<1$ case, it's easy to get $l_{jk}<l_{ij}+l_{ik}$. If further $l_{ik}<l_{ij}+l_{jk}$ and $l_{ij}<l_{ik}+l_{jk}$, then $\tilde{\theta}_i=\theta_i<\pi-\Lambda(I_{jk})$ by Claim \ref{claim-1}; If $l_{ik}\geq l_{ij}+l_{jk}$ or $l_{ij}\geq l_{ik}+l_{jk}$, then $\tilde{\theta}_i=0<\pi-\Lambda(I_{jk})$. Hence $\tilde{\theta}_i\leq\pi-\Lambda(I_{jk})$.\\

If $A_1$ is nonempty then $\tilde{\theta}^{jk}_i\leq\pi-\Lambda(I_{jk})$ by Claim \ref{claim-2} and hence
$$\sum_{i\in A, \{ijk\}\in A_1} \tilde{\theta}_i^{jk}\leq\sum_{(e,v)\in Lk(A)}\big(\pi-\Lambda(I_e)\big).$$
If $A_2$ is nonempty, then $\tilde{\theta}^{jk}_i+\tilde{\theta}^{ik}_j\leq\pi$ and hence
$$\sum_{i,j\in A, \{ijk\}\in A_2} (\tilde{\theta}_i^{jk}+\tilde{\theta}_j^{ik})\leq|A_2|\pi.$$
hence we always have
\begin{equation*}
\begin{aligned}
\sum_{i\in A}\widetilde{K}_i
=&2\pi|A|-
\Big(\sum_{i\in A, \{ijk\}\in A_1} \tilde{\theta}_i^{jk}+\sum_{i,\;j\in A, \{ijk\}\in A_2} (\tilde{\theta}_i^{jk}+\tilde{\theta}_j^{ik})+
\sum_{\{ijk\}\in A_3} (\tilde{\theta}_i^{jk}+\tilde{\theta}_j^{ik}+\tilde{\theta}_k^{ij})\Big)\\[6pt]
\geq&2\pi|A|-\sum_{(e,v)\in Lk(A)}\big(\pi-\Lambda(I_e)\big)-|A_2|\pi-|A_3|\pi\\
=&-\sum_{(e,v)\in Lk(A)}\big(\pi-\Lambda(I_e)\big)+2\pi\chi(F_A).
\end{aligned}
\end{equation*}\qed
\end{proof}

For a single triangle case, we can determine the shape of $K(\Omega)$ completely. Consider a triangle $\triangle123$ that is configured by three circles with three fixed non-negative numbers $I_{12}$, $I_{23}$ and $I_{13}$ as inversive distances. Recall the definition of the space of metrics $\Delta$ (in this setting $\Omega=\Delta$) and the angle map $\theta: \Delta\rightarrow \mathds{R}^3_{>0}$ in the beginning of Section \ref{section-prove}, we have
\begin{theorem}\label{thm-angle-range-one-triangle}
$\theta$ is a diffeomorphism from $\Delta'\triangleq\Delta\cap\big\{\prod_{i=1}^3r_i=1\big\}$ to $\mathcal{Z}$, where
\begin{equation*}\label{angle-range-asingletriangle}
\mathcal{Z}=\left\{(\theta_1, \theta_2, \theta_3)^T\in\mathds{R}^3\big|\theta_1+\theta_2+\theta_3=\pi; \,0<\theta_i<\pi-\Lambda(I_{jk}),\,\{i,j,k\}=\{1,2,3\}\right\}.
\end{equation*}
\end{theorem}
\begin{proof}
We follow the approach pioneered by Marden and Rodin \cite{Marden-Rodin}. Note $N=3$ in a single triangle setting, hence $\mathds{1}=(1,1,1)^T$, $r=(r_1,r_2,r_3)^T$ and $u=(u_1,u_2,u_3)^T$, where $u_i=\ln r_i$ for each $i\in\{1,2,3\}$. We divide the proof into four steps.

Step 1: we prove $\theta(\Delta)'\subset\mathcal{Z}$. This fact is essentially proved in Claim 1, that is, for all $r\in\Delta$, $0<\theta_i<\pi-\Lambda(I_{jk})$. Hence $\theta(\Delta')=\theta(\Delta)$ is contained in $\mathcal{Z}$.

Step 2: we prove $\theta:\Delta'\rightarrow\mathcal{Z}$ is injective. This fact is proved by Feng Luo in \cite{Luo1}, we give a new proof here with slight difference. Assuming $\bar{\theta}=(\bar{\theta}_1,\bar{\theta}_2,\bar{\theta}_3)^T$ is realized by some $\bar{r}=(\bar{r}_1,\bar{r}_2,\bar{r}_3)^T\in\Delta$. $\bar{u}=(\bar{u}_1,\bar{u}_2,\bar{u}_3)^T\in \ln\Delta$ is the corresponding metric in $u$-coordinate. If there is a metric $\bar{r}'\in\Delta$ with $\bar{r}\neq\bar{r}'$ that also realizes angle $\bar{\theta}$. Let $p=\ln\bar{r}$ and $q=\bar{r}'$, and define
\begin{equation}
W(u)\triangleq\int_{u_{0}}^{u}(\bar{\theta}_1-\tilde{\theta}_1)du_1+(\bar{\theta}_2-\tilde{\theta}_2)du_2+(\bar{\theta}_3-\tilde{\theta}_3)du_3, \,\,\,u\in\mathds{R}^3
\end{equation}
where $u_0\in\mathds{R}^3$ is arbitrary chosen. It's easy to get $W(u)=-\widetilde{F}_{123}(u)+(u-u_0)^T\bar{\theta}$.
Moreover, $W(u)\in C^1(\mathds{R}^3)$, $\nabla_uW=\bar{\theta}-\tilde{\theta}$ and $W(u)=W(u+t\mathds{1})$ for any $t\in\mathds{R}$. Obviously,
$\nabla_uW|_p=\nabla_uW|_q=0$. Let $\varphi(t)=W(p+t(q-p))$, by Lemma \ref{lemma-luo-essential}, $\widetilde{F}_{123}$ is concave and hence $\varphi(t)$ is convex and $C^1$ for $t\in\mathds{R}$. So $\varphi'(t)$ is monotone increasing. However, $\varphi'(0)=\varphi'(1)=0$, so $\varphi'(t)\equiv 0$ on $[0,1]$, therefore $\varphi(t)$ is constant on $[0,1]$ and there follows $W(p+t(q-p))\equiv W(p)$, for all $t\in[0,1]$. Denote $\Pi_p\triangleq\{u\in\mathds{R}^3|u^T\mathds{1}=p^T\mathds{1}\}$ as the plane in $\mathds{R}^3$ that passing $p$ and perpendicular to the direction $\mathds{1}$. Let $q^*$ be the projection of $q$ onto the plane $\Pi_p$. We claim that $q^*=p$, that is, $p$ and $q$ differs by a parallel move along direction $\mathds{1}$, this implies the original metric
$\bar{r}$ and $\bar{r}'$ differs by a scalar multiplication. We prove the claim by contradiction. Now suppose $q^*\neq p$. On one hand, $W(p)\equiv W(p+t(q-p))$ and $W(u)=W(u+t\mathds{1})$ implies $W(u)$ is a constant on the segment $\overline{pq^*}$. On the other hand, since $\ln\Delta$ is open, $p\in\ln\Delta$ and $q^*\neq p$, we can choose $\varepsilon$ such that $0<\varepsilon<\|q^*-p\|$ and $B(p,\varepsilon)\subset\ln\Delta$. Denote $B^*(p,\varepsilon)=B(p,\varepsilon)\cap\Pi_p$. On $B(p,\varepsilon)$, $W(u)$ is $C^{\infty}$-smooth, $Hess_uW$$=-\frac{\partial(\theta_1, \theta_2, \theta_3)}{\partial(u_1, u_2, u_3)}$ is positive semi-definite with null space $\{t\mathds{1}|t\in\mathds{R}\}$ by Lemma \ref{Lemma-Guo}. Since the null space of $Hess_uW$ is perpendicular to the plane $\Pi_p$, $Hess_uW$ is in fact positive definite when restrict to $\Pi_p$ (consider $W(u)$ as a function of two variables). This implies that $W|_{B^*(p,\varepsilon)}$ is strictly convex on $B^*(p,\varepsilon)\subset\Pi_p$ (or see Lemma \ref{lemma-guoren-F-123}). Now let $\psi(t)=W(p+t(q^*-p))$. $\psi(t)$ is strictly convex on $[0,\frac{\varepsilon}{\|q^*-p\|})$ since $W|_{B^*(p,\varepsilon)}$ is. Hence $\psi'(t)$ is strictly monotone increasing. Note $\psi'(0)=0$, hence $\psi'(t)>0$ for $t>0$, which implies that $\psi(t)$ is strictly increasing on $[0,\frac{\varepsilon}{\|q^*-p\|})$. Then we get a contradiction, since we had already proved $W(u)$ is constant on the segment $\overline{pq^*}$. Thus comes the claim $q^*=p$, which implies that $\bar{r}$ differs from $\bar{r}'$ by a scalar multiplication and the angle map $\theta$ is injective on $\Delta'$.

Step 3: we prove that when $r\in \Delta'$ tends to $\partial\Delta'$, then $\theta(r)$ tends to $\partial\mathcal{Z}$. Now we denote
$\Sigma_{12}\triangleq\{(r_1,r_2,0)|r_1,r_2>0\}$, $\Sigma_{13}\triangleq\{(r_1,0,r_3)|r_1,r_3>0\}$, $\Sigma_{23}\triangleq\{(0,r_2, r_3)|r_2,r_3>0\}$ and
$\partial_{ij}\Delta\triangleq\{r\in\mathds{R}^3_{>0}|l_{ij}=l_{ik}+l_{jk}\}$, where $\big\{\{ij\},\{ik\},\{jk\}\big\}=\big\{\{12\},\{13\},\{23\}\big\}$. According to the shape of $\partial\Delta'$, there are four cases to consider.
\begin{itemize}
  \item If $I_{12}, I_{13}, I_{23}\in[0,1]$, then $\Delta=\mathds{R}^3_{>0}$. For $\{i,j,k\}=\{1,2,3\}$, we need to prove
            \begin{equation}\label{step3-1}
            \lim_{r\in\Delta';\,(r_i, r_j, r_k)\rightarrow (0,0,+\infty)}\theta_k(r)=0,
            \end{equation}
            and for any $b\in(0,+\infty]$,
            \begin{equation}\label{step3-2}
            \lim_{r\in\Delta';\,(r_i, r_j, r_k)\rightarrow (0,b,+\infty)}\theta_i(r)=\pi-\Lambda(I_{jk}).
            \end{equation}
            Let $r^{(n)}\in\Delta'$ be a sequence of metric satisfying $(r_i^{(n)}, r_j^{(n)}, r_k^{(n)})\rightarrow (0,\,0,\,+\infty)$ or $(r_i^{(n)}, r_j^{(n)}, r_k^{(n)})\rightarrow (0,\,b,\,+\infty)$, then either $\theta_k(r^{(n)})\rightarrow 0$ or $\theta_i(r^{(n)})\rightarrow \pi-\Lambda(I_{jk})$.
  \item If $I_{ij}>1$, and $I_{ik}, I_{jk}\in[0,1]$, where $\big\{\{ij\},\{ik\},\{jk\}\big\}=\big\{\{12\},\{13\},\{23\}\big\}$. In this case, $\Delta$ is surrounded
            by three cone-like surfaces $\Sigma_{ik}$, $\Sigma_{jk}$ and $\partial_{ij}\Delta$. We need to prove (\ref{step3-1}) and (\ref{step3-2}) for $\{i,j,k\}=\{1,2,3\}$ too. We further need to prove for each point $(a,b,c)^T\in\partial_{ij}\Delta\cap\big\{\prod_{i=1}^3r_i=1\big\}$,
            \begin{equation}\label{step3-3}
            \lim_{r\in\Delta';\,r\rightarrow (a,b,c)}\theta_k(r)=\pi-\Lambda(I_{ij}).
            \end{equation}
  \item If $I_{ij}, I_{ik}>1$, and $I_{jk}\in[0,1]$, where $\big\{\{ij\},\{ik\},\{jk\}\big\}=\big\{\{12\},\{13\},\{23\}\big\}$. In this case, $\Delta$ is surrounded
            by three cone-like surfaces $\Sigma_{jk}$, $\partial_{ij}\Delta$ and $\partial_{ik}\Delta$. We need to prove (\ref{step3-1}) and (\ref{step3-2}) for $\{i,j,k\}=\{1,2,3\}$. We further need to prove (\ref{step3-3}) for each point $(a,b,c)^T\in\partial_{ij}\Delta\cap\big\{\prod_{i=1}^3r_i=1\big\}$ and
            \begin{equation}\label{step3-4}
            \lim_{r\in\Delta';\,r\rightarrow (a,b,c)}\theta_j(r)=\pi-\Lambda(I_{ik}).
            \end{equation}
            for each point $(a,b,c)^T\in\partial_{ik}\Delta\cap\big\{\prod_{i=1}^3r_i=1\big\}$.
  \item If $I_{12}, I_{13}, I_{23}>1$, then $\Delta$ is surrounded by three surfaces $\partial_{12}\Delta$, $\partial_{13}\Delta$ and
            $\partial_{23}\Delta$. We need to prove (\ref{step3-1}) for $\{i,j,k\}=\{1,2,3\}$. Furthermore, we need to prove (\ref{step3-3}) for each $\{ij\}\in\big\{\{12\},\{13\},\{23\}\big\}$ and each point $(a,b,c)^T\in\partial_{ij}\Delta\cap\big\{\prod_{i=1}^3r_i=1\big\}$.
\end{itemize}

The limit (\ref{step3-4}) is equivalent to limit (\ref{step3-3}). However, the limits (\ref{step3-1})-(\ref{step3-3}) can be proved by the law of cosines directly. We omit the details here, since it is almost the same with the proof of Lemma \ref{Lemma-limit-xita}.

Step 4: We prove the theorem finally. Step 3 implies that the angle map $\theta$ is proper. On one hand, $\theta$ is a closed map, by a pure topological result which says that a proper map $f$ from $X$ to $Y$ is a closed map, if $X$ is Hausdorff and $Y$ is locally compact Hausdorff. On the other hand, both $\Delta'$ and $\mathcal{Z}$ are homeomorphic to $\mathds{R}^2$, hence $\theta$ is an open map by the invariance of domain theorem. Therefore $\theta(\Delta')$ is a closed and open nonempty subset of $\mathcal{Z}$. Note $\mathcal{Z}$ is connected, then $\theta(\Delta')=\mathcal{Z}$, which implies that $\theta:\Delta'\rightarrow\mathcal{Z}$ is a diffeomorphism.\qed
\end{proof}

By Theorem \ref{thm-angle-range-one-triangle}, for one single triangle setting we solved $\theta(\Omega)=\theta(\Delta)$ completely. However, we can not combinatorially combine Theorem \ref{thm-angle-range-one-triangle} together to get the exactly range of $K(\Omega)$, although we can really do so for Andreev-Thurston's $(M, \mathcal{T}, \Phi)$ setting. Consider the general triangulation $(M,\mathcal{T}, I)$ with inversive distance $I\geq0$, suppose $\triangle123$ is a triangle in $\mathcal{T}$. For convenience, we suppose $I_{23}>1$ while other inversive distance are all in $[0,1]$. For any fixed $\bar{r}_2, \bar{r}_3, \cdots, \bar{r}_N>0$, let $\bar{r}_1>0$ be the unique positive solution of equation $$\sqrt{\bar{r}_1^2+\bar{r}_2^2+2\bar{r}_1\bar{r}_2I_{12}}+\sqrt{\bar{r}_1^2+\bar{r}_3^2+2\bar{r}_1\bar{r}_3I_{13}}=\sqrt{\bar{r}_2^2+\bar{r}_3^2+2\bar{r}_2\bar{r}_3I_{23}}.$$ $\Omega$ is the space of all inversive circle packing metrics. It's easy to see $(\bar{r}_1,\bar{r}_2,\cdots,\bar{r}_N)\in\partial\Omega$. Then $\lim\limits_{r_1\rightarrow \bar{r}_1}\theta_1^{23}(r_1,\bar{r}_2,\cdots,\bar{r}_N)=\pi-\Lambda(I_{23})$, while for the other triangle $\{1ij\}$, the inner angle $\theta_1^{ij}$ satisfies $\lim\limits_{r_1\rightarrow \bar{r}_1}\theta_1^{ij}(r_1,\bar{r}_2,\cdots,\bar{r}_N)<\pi-\Lambda(I_{ij})$, hence
$$\lim\limits_{r_1\rightarrow \bar{r}_1}K_1>-\sum_{(e,v)\in Lk(\{1\})}\big(\pi-\Lambda(I_{e})\big)+2\pi\chi(F_{\{1\}}).$$
This shows that even if $r$ tends to the boundary of $\Omega$, the curvature $K_1$ may not tends to the boundary of $Y_A$ in (\ref{def-Y-A}) for $A={\{1\}}$.\\

\noindent \textbf{Acknowledgements}: Both authors would like to thank Professor Gang Tian for constant encouragement. The research is supported by National Natural Science Foundation of China under grant no.11501027. The first author would also like to thank Professor Feng Luo, Ren Guo for many helpful conversations.

Huabin Ge: hbge@bjtu.edu.cn

Department of Mathematics, Beijing Jiaotong University, Beijing 100044, P.R. China\\

Wenshuai Jiang: jiangwenshuai@pku.edu.cn

BICMR and SMS of Peking University, Yiheyuan Road 5, Beijing 100871, P.R. China

\begin{thebibliography}{50}
\setlength{\itemsep}{-2pt} \small
\bibitem{Andreev1} E. M. Andreev, \emph{On convex polyhedra in Loba$\check{c}$evski$\check{i}$ spaces}, Math. USSR-Sb., 10 (1970), 412-440.

\bibitem{Andreev2} E. M. Andreev, \emph{On convex polyhedra of finite volume in Loba$\check{c}$evski$\check{i}$ spaces}, Math. USSR-Sb., 12 (1970), 255-259.

\bibitem{Bobenko} A. Bobenko, U. Pinkall, B. Springborn, \emph{Discrete conformal maps and ideal hyperbolic polyhedra}, Geom. Topol., 19 (2015), 2155-2215.

\bibitem{Bowers-Hurdal} P. L. Bowers, M. K. Hurdal, \emph{Planar conformal mappings of piecewise flat surfaces}, from: ``Visualization and mathematics III", (H-C Hege, K Polthier, editors), Math. Vis., Springer, Berlin, (2003) 3每34.

\bibitem{Bowers-Stephenson} P. L. Bowers, K. Stephenson, \emph{Uniformizing dessins and Bely$\check{i}$ maps via circle packing}, Mem. Amer. Math. Soc. 170, no. 805, Amer. Math. Soc., (2004).

\bibitem{CL1} B. Chow, F. Luo, \emph{Combinatorial Ricci flows on surfaces}, J. Differential Geometry, 63 (2003), 97-129.

\bibitem{Colindev} Y. C. de Verdi\`{e}re, \emph{Un principe variationnel pour les empilements de cercles}, Invent. Math., 104(3) (1991), 655-669.

\bibitem{Ge-Jiang1} H. Ge, W. Jiang, \emph{On the deformation of discrete conformal factors on surfaces}, in preparation.

\bibitem{Ge-Xu1} H. Ge, Xu. Xu, \emph{2-dimensional combinatorial Calabi flow in hyperbolic background geometry}, Differ. Geom. Appl., 47 (2016), 86-98.

\bibitem{Glickenstein1} D. Glickenstein, \emph{A combinatorial Yamabe flow in three dimensions}, Topology 44 (2005), No. 4, 791-808.

\bibitem{Guoren} R. Guo, \emph{Local rigidity of inversive distance circle packing}, Trans. Amer. Math. Soc., 363 (2011), 4757每4776.

\bibitem{Hezhengxu1} Z. He, \emph{Rigidity of infinite disk patterns}, Ann. of Math., 149 (1999), 1-33.

\bibitem{Hurdal-Stephenson} M. K. Hurdal, K. Stephenson, \emph{Discrete conformal methods for cortical brain flattening}, Neuroimage, 45 (2009), S86-S98.

\bibitem{Koebe1} P. Koebe, \emph{Kontaktprobleme der konformen Abbildung}, Ber. S$\ddot{a}$chs. Akad. Wiss. Leipzig, Math.-phys. Kl. 88 (1936), 141-164.

\bibitem{Luo0} F. Luo, \emph{Combinatorial Yamabe flow on surfaces}, Commun. Contemp. Math., 6(5) (2004), 765-780.

\bibitem{Luo1} F. Luo, \emph{Rigidity of polyhedral surfaces, III}, Geom. Topol. 15 (2011), 2299-2319.

\bibitem{Luo2} F. Luo, T. Yang, \emph{Volume and rigidity of hyperbolic polyhedral 3-manifolds}, arXiv:1404.5365v2 [math.GT].

\bibitem{Marden-Rodin} A. Marden, B. Rodin, \emph{On Thurston＊s formulation and proof of Andreev＊s theorem}, Computational methods and function theory (Valparaso, 1989), 103每115, Lecture Notes in Math., 1435, Springer, Berlin, 1990.

\bibitem{Ri} I. Rivin, \emph{Euclidean structures on simplicial surfaces and hyperbolic volume}, Ann. of Math., 139 (1994), 553每580.

\bibitem{Stephenson} K. Stephenson, \emph{Introduction to circle packing. The theory of discrete analytic functions}, Cambridge University Press, Cambridge, 2005.

\bibitem{T1} W. Thurston, \emph{Geometry and topology of 3-manifolds}, Princeton lecture notes, 1976, \href{http://www.msri.org/publications/books/gt3m}{http://www.msri.org/publications/books/gt3m}.

\bibitem{Zhang} M. Zhang, R. Guo, W. Zhang, F. Luo, S-T. Yau, X. Gu, \emph{The unified discrete surface Ricci flow}, Graphical Models, 76(5) (2014), 321每339.
\end{thebibliography}
\end{document}